\documentclass[12pt,a4paper]{amsart}

\usepackage[T1]{fontenc}
\usepackage{amssymb,amsmath,latexsym,amsthm,paralist,a4,mathrsfs,dsfont}
\usepackage{stmaryrd}

\usepackage{tikz-cd}
\usetikzlibrary{decorations.markings}

\makeatletter
\tikzcdset{
  open/.code     = {\tikzcdset{hook, circled};},
  closed/.code   = {\tikzcdset{hook, slashed};},
  open'/.code    = {\tikzcdset{hook', circled};},
  closed'/.code  = {\tikzcdset{hook', slashed};},
  circled/.code  = {\tikzcdset{markwith = {\draw (0,0) circle (.375ex);}};},
  slashed/.code  = {\tikzcdset{markwith = {\draw[-] (-.4ex,-.4ex) -- (.4ex,.4ex);}};},
  markwith/.code ={
    \pgfutil@ifundefined%
    {tikz@library@decorations.markings@loaded}%
    {\pgfutil@packageerror{tikz-cd}{You need to say %
      \string\usetikzlibrary{decorations.markings} to use arrows with markings}{}}{}%
    \pgfkeysalso{/tikz/postaction = {
      /tikz/decorate,
      /tikz/decoration={markings, mark = at position 0.5 with {#1}}}
    }
  },
}
\makeatother

\usepackage[headings]{fullpage}
\usetikzlibrary{babel}
\usetikzlibrary{positioning,calc}
\usepackage{varioref}
\usepackage[pdfpagelabels, pdftex]{hyperref}
\hypersetup{
  pdftitle={Tame and strongly \'etale cohomology of curves},
  pdfauthor={Katharina H\"{u}bner},
  pdfsubject={},
  pdfkeywords={},
  colorlinks=true,    
  linkcolor=red,     
  citecolor=blue,     
  filecolor=blue,      
  urlcolor=blue,       
  breaklinks=true,
  bookmarksopen=true,
  bookmarksnumbered=true,
  pdfpagemode=UseOutlines,
  plainpages=false,
  unicode=true
  }
\usepackage[capitalise]{cleveref}
\usepackage{changes}  
\usepackage{bbold} 
\newcommand{\Gal}{\mathrm{Gal}}

\newcommand{\Hom}{\mathrm{Hom}}

\newcommand{\cd}{\mathop{cd}\nolimits}

\newcommand{\m}{\mathfrak{m}}

\newcommand{\Z}{\mathds Z}

\newcommand{\Q}{\mathds Q}

\newcommand{\F}{\mathds F}

\newcommand{\GG}{\mathbb{G}}

\newcommand{\et}{\textit{\'et}}

\newcommand{\tr}{\mathrm{tr}}

\newcommand{\Mod}{\mathrm{Mod}}

\newcommand{\Ext}{\mathrm{Ext}}

\renewcommand{\P}{{\mathbb P}}

\newcommand{\cO}{{\mathscr O}}

\newcommand{\liso}{\mathrel{\hbox{$\longrightarrow$} \kern-2.4ex\lower-1ex\hbox{$\scriptstyle\sim$}\kern1.7ex}}

\newcommand{\set}{\textit{s\'et}}

\makeatletter
\newcommand{\zerounderset}[3][\mathord]{%
  #1{\vtop{
    \let\\\cr
    \baselineskip\z@skip\lineskip.25ex
    \ialign{\hidewidth$##$\hidewidth\crcr
      \omit$#3$\cr
      #2\crcr
    }%
  }}%
}
\makeatother

\newtheoremstyle{alexthm}
  {}
  {}
  {\sl }
  {}
  {\bf}
  {.}
  {.5em}
  {}
\theoremstyle{alexthm}

\newtheorem{theorem}{Theorem}[section]
\newtheorem*{theorem*}{Theorem}
\newtheorem{corollary}[theorem]{Corollary}
\newtheorem{proposition}[theorem]{Proposition}
\newtheorem{lemma}[theorem]{Lemma}
\newtheorem*{lemma*}{Lemma}

\theoremstyle{remark}

\newtheoremstyle{alexdef}
  {}
  {}
  {\rm }
  {}
  {\bf}
  {.}
  {.5em}
  {}
\theoremstyle{alexdef}
\newtheorem*{example*}{Example}

\newtheorem{remark}[theorem]{Remark}

\newtheorem{definition}[theorem]{Definition}

\DeclareMathOperator{\Spec}{\textit{Spec}}

\DeclareMathOperator{\Spa}{\textit{Spa}}

\DeclareMathOperator{\supp}{\textrm{Supp}}

\DeclareMathOperator{\ch}{char}

\DeclareMathOperator{\Shv}{\mathrm{Shv}}

\definecolor{darklimegreen}{RGB}{31,142,8}

\title{Tame and strongly \'etale cohomology of curves}
\author{Katharina H\"{u}bner}
\date{}
\begin{document}

\begin{abstract}
 For a curve~$C$ over a perfect field~$k$ of characteristic $p > 0$ we study the tame cohomology of $X = \Spa(C,k)$ introduced in \cite{HueAd}.
 We prove that the tame cohomology groups of~$X$ with $p$-torsion coefficients satisfy cohomological purity (which is not true in full generality for the \'etale cohomology).
 Using purity we show Poincar\'e duality for the tame cohomology of~$X$ with $p$-torsion coefficients.
\end{abstract}

\maketitle

\section{Introduction}

The tame cohomology of an $S$-scheme $X$ is  defined as the tame cohomology (introduced in \cite{HueAd}) of its associated discretely ringed adic space $\Spa(X,S)$.
For torsion coefficients prime to $\ch(S)$, it coincides with \'{e}tale cohomology, see \cite{HueAd}, \S 7.
For $p$-torsion coefficients in characteristic~$p$, however, tame cohomology is expected to be better behaved than \'etale cohomology.
For instance, smooth base change, cohomological purity, and homotopy invariance should hold for tame cohomology with $p$-torsion coefficients.
Assuming resolution of singularities, a special case of cohomological purity was already proved in \cite{HueAd}, Corollary~13.5.
A consequence of purity is homotopy invariance (\cite{HueAd}, Corollary~13.6).

Let us first explain what form cohomological purity is expected to take in characteristic~$p$.
The role played by the roots of unity for invertible coefficients is now taken by the logarithmic de Rham Witt sheaves.
For simplicity we restrict to sheaves of $\F_p$-modules (instead of $\Z/p^n\Z$ for $n \ge 1$).
The logarithmic de Rham sheaf~$\nu(r)$ for $r \ge 0$ on a scheme~$X$ smooth over a perfect field~$k$ of characteristic $p > 0$ is defined by the short exact sequence
\[
 0 \to \nu(r) \to \Omega^r_{X/k,d=0} \overset{C-1}{\longrightarrow} \Omega^r_{X/k} \to 0,
\]
where~$C$ denotes the Cartier operator (see \cite{Milne76}, Lemma~1.1) and ``$d=0$'' stands for closed forms.
It is a sheaf of $\F_p$-modules and equals $\F_p$ in case $r=0$.
For $r<0$ we set $\nu(r) = 0$.
Then cohomological purity for a smooth closed subscheme~$Z$ of~$X$ of codimension~$c$ is expected to take the form
\[
 H^{n-c}(Z,\nu(r-c)) \overset{\sim}{\longrightarrow} H^n_Z(X,\nu(r)),
\]
where~``$H$'' stands for an appropriate cohomology theory (see \cite{Milne86}, \S 2).
In \'etale cohomology the above statement is true in general only for $r \ge c$ (see \cite{Milne86}, Corollary~2.2 and Remark~2.4 and \cite{Mos99}, \S 2).
If~``$H$'' stands for tame cohomology, the statement is expected to hold for all~$r$.
In case $r=0$ this has been shown in \cite{HueAd}, Corollary~13.5 under the assumption of resolution of singularities.
Note that in this case the assertion does not involve logarithmic de Rham sheaves.
It just says that $H^n_Z(X,\F_p) = 0$ for all~$n$.

In order to show cohomological purity in general, it would be necessary to develop more machinery such as a theory of logarithmic deRham-Witt complexes in this setting.
This is a task to be tackled in the long run.
For curves, however, it is possible to prove purity for the tame cohomology without any further tools.
We fix a smooth curve~$C$ over a perfect field~$k$ and set $X = \Spa(C,k)$.
We denote the tame site of~$X$ by~$X_t$.
Then the result is the following (see \cref{purity-sec}:

\begin{theorem}
 Let $c \in C$ be a closed point.
 Set $x = \Spa(c,c)$, such that we have a closed immersion $i:x \to X$.
 Then
 \[
  H^r(x_t,\nu(s)) \cong H^{r+1}_x(X_t,\nu(s+1))
 \]
 for all $s,r \in \Z$.
 In particular,
 \[
  H^r(x_t,\F_p) \cong H^{r+1}_x(X_t,\nu(1))
 \]
 for all $r \ge 0$ and
 \[
  H^r_x(X_t,\nu(s)) = 0
 \]
 for all $r \ge 0$ and $s \ne 1$.
\end{theorem}

The second aim of the article  is to prove Poincar\'e duality for the tame site.
As a preparation we introduce horizontally constructible sheaves on discretely valued adic spaces in \cref{constructible_sheaves}.
These are constructible sheaves that are locally constant on constructible subsets with respect to the \emph{Zariski} topology.
They turn out to be the sheaves for which Poincar\'e duality holds.
After computing the cohomological dimension of horizontally constructible sheaves on curves in \cref{cohdim-sec} we are prepared for proving Poincar\'e duality:

\begin{theorem}
 Let~$C$ be a smooth curve over a perfect field~$k$ and set $X= \Spa(C,k)$.
 For every horizontally constructible sheaf of $\F_p$-modules~$\mathcal{F}$ on~$X_t$ there is a natural quasi-isomorphism
 \[
  R\Hom_X(\mathcal{F},\nu_X(1))[1] \to R\Hom_k(R\pi_!\mathcal{F},\F_p)
 \]
 in $D^+(\F_p)$.
\end{theorem}

The proof is a formal consequence of cohomological purity and Poincar\'e duality for projective curves (see \cite{Milne76}).
It is given in \cref{duality-sec}.

\medskip\noindent
\emph{Acknowledgement:} The author wants to thank Alexander Schmidt for thoroughly proofreading this preprint. His feedback led to many improvements.

\section{Preliminaries}

In this section we collect some basic results on the tame and (strongly) \'etale sites of adic spaces.
We will tacitly assume that all adic spaces are of the type examined by Huber:
They are either strongly noetherian and analytic or the quasi-compact sections of the structure sheaf have a noetherian ring of definition.
This also includes all discretely ringed adic spaces.

Recall (\cite{Hu96}, Definition~1.6.5) that a morphism of adic spaces $f:Y \to X$ is \emph{\'etale} if it is locally of finite presentation and for any solid arrow diagram
\[
 \begin{tikzcd}
  \Spa(A,A^+)/I	\ar[r]	\ar[d]		& Y	\ar[d,"f"]	\\
  \Spa(A,A^+)	\ar[r]	\ar[ur,dotted]	& X,
 \end{tikzcd}
\]
where~$I$ is an ideal of~$A$ with $I^2=0$, a unique dotted arrow exists making the diagram commutative.
The morphism~$f$ is called \emph{tame} (\emph{strongly \'etale}) if in addition to being \'etale it has the following property for every point $y \in Y$ with image $x \in X$:
The residue field extension $k(y)|k(x)$ is tamely ramified (unramified) with respect to the valuation corresponding to~$y$ (see \cite{HueAd}, Definition~3.1).

The \emph{\'etale site}~$X_{\et}$ of an adic space~$X$ has as its underlying category the category of all \'etale morphisms $Y \to X$ and coverings are surjective families.
The \emph{tame site}~$X_t$ and the \emph{strongly \'etale site}~$X_{\set}$ are the full subcategories of~$X_{\et}$ defined by taking all tame, respectively strongly \'etale morphisms $Y \to X$.
Throughout this article we will write~$\tau$ for one of these three topologies $\{\et,\set,t\}$.
For instance, we will prove assertions about~$X_{\tau}$ meaning that they are true for the \'etale, the tame and the strongly \'etale site.
If~$\mathcal{F}$ is a sheaf of abelian groups on~$X_{\tau}$, we write $H^n(X,\mathcal{F})$ for its cohomology groups (instead of $H^n(X_{\tau},\mathcal{F})$) whenever this does not cause confusion.

Recall from \cite{HueAd}, Definition~6.1 that a Huber pair $(A,A^+)$ is \emph{local} if $A$ and~$A^+$ are local,
 $A^+$ is the valuation subring of~$A$ associated with a valuation whose support is the maximal ideal of~$A$, and the maximal ideal~$\m^+$ of~$A^+$ is open.
It is \emph{henselian} (strongly henselian, strictly henselian if it is local and~$A^+$ is henselian ($A^+$ is strictly henselian, $A$ is strictly henselian).
A strongly henselian Huber pair $(A,A^+)$ is \emph{tamely henselian}
 if the value group of the associated valuation~$v$ is a $\Z[\frac{1}{p}]$-module, where~$p$ denotes the residue characteristic of~$v$.
We call an adic space local, henselian, strongly henselian, tamely henselian, or strictly henselian if it is the adic spectrum of a Huber pair with the respective property.
Given an adic space~$X$ and a point $x \in X$ we obtain a local adic space, the \emph{localization}~$X_x$, by taking the intersection of all open neighborhoods of~$x$.
We obtain a henselian adic space by taking the limit $\lim Y$ over all pointed \'etale morphisms $(Y,y) \to (X,x)$ such that $k(y)^+ = k(x)^+$.
It is called the \emph{henselization} of~$X$ at~$x$ and we denote it $X_x^h$.
For $\tau \in \{\set,t,\et\}$ and a geometric point~$\bar{x}$ with respect to~$\tau$ we define the $\tau$-localization (or just localization) $X_{\bar{x}}^{\tau}$ at~$\bar{x}$
 to be the limit $\lim Y$ over all $\tau$-morphisms $Y \to X$ together with a lift of~$\bar{x}$ to~$Y$.
If $\tau = \set$, $X_{\bar{x}}^{\tau}$ is strongly henselian, if $\tau = t$, it is tamely henselian, and if $\tau = \et$, it is strictly henselian.
According to wether $\tau = \set,t$, or~$\et$ we also call~$X_{\bar{x}}^{\tau}$ the strong henselization, tame henselization, or strict henselization, respectively.
See \cite{HueAd}, \S 6 for more detail.


\section{Cohomological purity} \label{purity-sec}

In this section we prove a positive characteristic version of cohomological purity.
More precisely, we show cohomological purity for $p$-torsion sheaves on a curve of characteristic $p > 0$.
In this setting the Tate twist by $\ell$-th roots of unity is replaced by a tensor product with the logarithmic de Rham sheaves $\nu(r)$.
Let us first recall their definition.
Let~$X$ be a scheme of finite type over a perfect field~$k$ of characteristic $p > 0$.
We consider the sheaves of differentials $\Omega^r_{X/k}$ and the sheaves of closed differentials $\Omega^r_{X/k,d=0}$ on~$X$.
For every $r \ge 0$ there is a Cartier operator
\[
 C : \Omega^r_{X/k,d=0} \to \Omega^r_{X/k}
\]
defined uniquely by the properties
\begin{enumerate}[(a)]
 \item $C(1) = 1$;
 \item $C(f^p\omega) = fC(\omega)$ for $f \in \cO_X$, $\omega \in \Omega^r_{X/k,d=0}$;
 \item $C(\omega \wedge \omega')= C(\omega) \wedge C(\omega')$ for $\omega, \omega'$ closed;
 \item $C(\omega) = 0$ if and only if~$\omega$ is exact;
 \item $C(f^{p-1}df) = df$ for $f \in \cO_X$;
\end{enumerate}
(see \cite{Milne76}, Lemma~1.1).
The logarithmic de Rham sheaf~$\nu(r)$ is defined to be the kernel of
\[
 1-C : \Omega^r_{X/k,d=0} \to \Omega^r_{X/k},
\]
which in fact is surjective with respect to the \'etale topology (see \cite{Milne76}, Lemma~1.3), i.e., we obtain a short exact sequences of \'etale sheaves
\[
 0 \to \nu(r) \to \Omega^r_{X/k,d=0} \overset{1-C}{\longrightarrow} \Omega^r_{X/k} \to 0.
\]

Assume now that~$X$ is smooth and quasi-projective over~$k$
In \cite{Milne86}, section~2, Milne constructs for every smooth closed subscheme $Z \subset X$ of codimension~$c$ and for every integer $r$ a cycle map
\[
 H^0(Z,\nu(r)) \to H^c_Z(X,\nu(r+c)),
\]
which is an isomorphism by \cite{Milne86}, Corollary~2.2.
Moreover, the same corollary states that
\[
 H^i_Z(X,\nu(r)) = 0
\]
for $i < c$.
By the general philosophy of cohomological purity we would expect isomorphisms
\begin{equation} \label{puritymap}
 H^i(Z,\nu(r)) \overset{\sim}{\to} H^{c+i}_Z(X,\nu(r+c))
\end{equation}
for every integer~$i$.
Milne's results show that this is true for $i \le c$ but it is wrong, in general, for $i > c$ (see \cite{Milne86}, Remark~2.4).

In this section we will show, in the case that~$X$ is a curve over an algebraically closed field, that replacing the \'etale cohomology by the tame cohomology, we obtain indeed purity isomorphisms of the form (\ref{puritymap}) as above for all integers~$i$.
Let us fix notation.
For the rest of this section~$k$ is a perfect field of characteristic $p > 0$.
We consider a curve~$C$ over~$k$ (i.e., a one-dimensional scheme of finite type over~$k$) and set $X = \Spa(C,k)$.
Pulling back the sheaf $\nu(r)$ on~$C$ via the support map
\[
 \supp: X = \Spa(C,k) \to C
\]
we obtain an \'etale sheaf on~$X$, which we can also view as a tame or strongly \'etale sheaf.
We denote it again by $\nu(r)$.

\begin{lemma} \label{cohdimsupp}
 Suppose that~$k$ is algebraically closed.
 Let $c \in C$ be a closed point.
 Set $x = \Spa(c,c)$, such that we have a closed immersion $x \to X$.
 Let $\tau \in \{t,\set\}$ and~$\mathcal{F}$ be a sheaf on~$X_{\tau}$.
 Assume that~$\mathcal{F}$ is a $p$-torsion sheaf in case~$\tau$ is the tame topology.
 Then
 \[
  H^i_x(X,\mathcal{F}) = 0
 \]
 for all $i \ne 0,1$ and we have an exact sequence
 \[
  0 \to H^0_x(X,\mathcal{F}) \to \mathcal{F}(\Spa(A,A)) \to \mathcal{F}(\Spa(K,A)) \to H^1_x(X,\mathcal{F}) \to 0,
 \]
 where $A$ is the henselization of the reduced local ring of~$C$ at~$c$ and~$K$ is its total ring of fractions.
\end{lemma}

\begin{proof}
 It suffices to prove the assertion for the strongly \'etale topology.
 For the tame topology it then follows by \cite{HueAd}, Proposition~9.5.
 Without loss of generality we may assume that~$C$ is reduced (see \cite{HueAd}, Proposition~8.1).
 The henselization of~$X$ at~$x$ is given by $\Spa(A,A)$.
 Hence we obtain by \cite{HueAd}, Proposition~8.3 that
 \[
  H^i_x(X_{\set},\mathcal{F}) = H^i_x(\Spa(A,A)_{\set},\mathcal{F}).
 \]
 The excision sequence for $x \hookrightarrow \Spa(A,A)$ reads
 \[
  0 \to H^0_x(X_{\set},\mathcal{F}) \to \mathcal{F}(\Spa(A,A)) \to \mathcal{F}(\Spa(K,A)) \to H^1_x(X_{\set},\mathcal{F}) \to \ldots
 \]
 The Huber pairs $(A,A)$ and $(K,A)$ are both strongly henselian.
 Therefore,
 \[
  H^i(\Spa(A,A)_{\set},\mathcal{F}) = H^i(\Spa(K,A)_{\set},\mathcal{F}) = 0.
 \]
 for $i \ge 1$.
\end{proof}

\begin{lemma} \label{calc_cohsupp}
 With notation as in \cref{cohdimsupp} we have that
 \[
  \nu(1)(\Spa(A,A)) \to \nu(1)(\Spa(K,A))
 \]
 is injective, i.e., $H^0_x(X,\nu(1)) = 0$ and
 \[
  0 \to \nu(1)(\Spa(A,A)) \to \nu(1)(\Spa(K,A)) \to H^1_x(X_{\tau},\nu(1)) \to 0
 \]
 is exact for $\tau \in \{t,\set\}$.
\end{lemma}

\begin{proof}
 In view of \cref{cohdimsupp} all we need to show is that $\nu(1)(\Spa(A,A)) \to \nu(1)(\Spa(K,A))$ is injective.
 By definition, $\nu(1)$ is subsheaf of $\Omega^1_{\Spa(A,A)}$, which is locally free of rank~$1$ over the structure sheaf.
 As $\Spa(A,A)$ is local, $\Omega^1_{\Spa(A,A)}$ is isomorphic to $\cO_{\Spa(A,A)}$ and thus the map
 \[
  \Omega^1_{\Spa(A,A)}(\Spa(A,A)) \to \Omega^1_{\Spa(A,A)}(\Spa(K,A))
 \]
 identifies with the natural inclusion of $A$ into $K$.
 In particular, it is injective.
 Hence, also the corresponding map for $\nu(1)$ is injective.
\end{proof}

\begin{proposition} \label{purity_sheaf}
 Let $c \in C$ be a closed point.
 Set $x = \Spa(c,c)$, such that we have a closed immersion $i:x \to X$.
 Then, for $\tau \in \{t,\set\}$,
 \[
  R^si^!\nu(1) = 0
 \]
 for $s \ne 1$ and $R^1i^!\nu(1)$ is isomorphic to the constant sheaf\/ $\F_p$ on~$x$.
\end{proposition}

\begin{proof}
 Let~$\bar{x}$ be a geometric point lying over~$x$ corresponding to a separable closure $\bar{k}$ of the residue field $k(x)$ at~$x$.
 Note that over~$x$ geometric points for the tame and the strongly \'etale topology are the same.
 Denote by~$A$ the strict henselization of~$\cO_C$ at $\Spec \bar{k} \to C$.
 Then $\Spa(A,A)$ is the strongly \'etale localization and also the tame localization of~$X$ at~$\bar{x}$.
 As $\tau$-cohomology commutes with limits (\cite{HueAd}, Proposition~5.3), the stalk of $R^si^!\nu(1)$ is
 \[
  R^si^!\nu(1)_{\bar{x}} = H^s_{\bar{x}}(\Spa(A,A),\nu(1)).
 \]
 Applying \cref{calc_cohsupp} to the base change of~$C$ to a separable closure of~$k$ (chosen compatibly with $\bar{k})$ we obtain that
 \[
  R^si^!\nu(1) = 0.
 \]
 for $s \ne 1$ and an exact sequence
 \[
  0 \to \nu(1)(\Spa(A,A)) \to \nu(1)(\Spa(K,A)) \to H^1_{\bar{x}}(\Spa(A,A),\nu(1)) \to 0.
 \]
 The sheaf $\nu(1)$ on~$X$ is actually the pullback of the corresponding sheaf $\nu(1)$ on $C$ via the support map $X \to C$.
 Therefore, the above exact sequence identifies with the exact sequence in the proof of Proposition~2.1 in \cite{Milne86} (top of page 316)
 \[
  0 \to \nu(1)(\Spec A) \to \nu(1)(\Spec K) \to H^1_{\bar{c}}(\Spec A,\nu(1)) \to 0,
 \]
 where $\bar{c} = \Spec \bar{k}$ is the support of~$\bar{x}$.
 Choose a uniformizer~$t$ of~$A$.
 In the proof of loc.\ cit.\ Milne shows that the map
 \begin{align*}
  \F_p	& \longrightarrow \nu(1)(\Spec K)/\nu(1)(\Spec A)	\\
  a	& \mapsto a\frac{dt}{t}
 \end{align*}
 is an isomorphism and is independent of the choice of~$t$.
 Hence,
 \[
  H^1_{\bar{x}}(\Spa(A,A),\nu(1)) \cong \F_p.
 \]
 In order to complete the proof it remains to show that the action of the absolute Galois group of~$k(x)$ on $H^1_{\bar{x}}(\Spa(A,A),\nu(1))$ is trivial.
 The isomorphism
 \[
  H^1_{\bar{x}}(\Spa(A,A),\nu(1)) \cong \nu(1)(\Spec K)/\nu(1)(\Spec A)
 \]
 is Galois equivariant and as we have seen above every class in $\nu(1)(\Spec K)/\nu(1)(\Spec A)$ has a representative of the form $a\, dt/t$ with $a \in \F_p$.
 Moreover, we can choose the uniformizer~$t$ to be Galois invariant.
 Then, $a dt/t$ is also Galois invariant.
\end{proof}

\begin{corollary} \label{purity}
 With notation as in \cref{purity_sheaf} we have that
 \[
  H^r(x_{\tau},\nu(s)) \cong H^{r+1}_x(X_{\tau},\nu(s+1))
 \]
 for all $s,r \in \Z$ and $\tau \in \{t,\set\}$.
 In particular,
 \[
  H^r(x_{\tau},\F_p) \cong H^{r+1}_x(X_{\tau},\nu(1))
 \]
 for all $r \ge 0$ and
 \[
  H^r_x(X_{\tau},\nu(s)) = 0
 \]
 for all $r \ge 0$ and $s \ne 1$.
\end{corollary}

\begin{proof}
 As~$C$ is a curve, $\Omega^s_{C/k}$ is nontrivial only for $s=0$ and $s=1$.
 Hence, $\nu(s)$ is nontrivial on~$X$ only for $s=0$ and $s=1$.
 Similarly, $\nu(s)$ is nontrivial on~$x$ only for $s=0$.
 If $s=0$, the assertion follows directly from \cref{cohdimsupp} as $\nu(0)$ is the constant sheaf $\F_p$.
 For $s=1$ consider the spectral sequence
 \[
  H^m(x_{\tau},R^ni^!\nu(1)) \Rightarrow H^{m+n}_x(X_{\tau},\nu(1)).
 \]
 By \cref{purity_sheaf}, we have
 \[
  R^ni^!\nu(1) = \begin{cases}
                  \F_p	& n=1	\\
                  0	& n\ne 1.
                 \end{cases}
 \]
 This gives the desired isomorphism.
\end{proof}

\begin{remark}
 For the \'etale topology purity does not hold in full generality.
 A partial result was obtained by Milne in \cite{Milne86}, \S 2 and by Moser in \cite{Mos99}, \S 2.
 If~$X$ is a smooth scheme of dimension~$d$ over a perfect field~$k$ and $i:Z \hookrightarrow X$ a smooth closed subscheme of codimension~$c$, then there is a canonical quasi-isomorphism
 \[
  \nu_Z(d-c)[-c] \cong Ri^!\nu_X(d)
 \]
 in $D(Z)$ (see \cite{Mos99}, Corollary to Theorem~2.4 in combination with (1.6)).
 So we have purity only for the sheaf $\nu_X(d)$ and not for all $\nu_X(r)$.
 In order to see that in general, purity does not hold for $r < d$ consider the following situation:
 Let~$k$ be an algebraically closed field, $X = \P_k^1$, $Z = \{\infty\}$, and $s=0$.
 Then $H^n(X,\F_p) = 0$ for $n \ge 1$ but $H^1(X-Z,\F_p)$ is infinite dimensional.
 In particular, the cohomology group with support in~$Z$, $H^2_Z(X,\F_p)$ is not zero.
 But this is what would be expected from purity.
\end{remark}

%
%
%
%

\section{Horizontally constructible sheaves} \label{constructible_sheaves}
The second goal of this article is to use cohomological purity in order to prove Poincar\'e duality.
However, for a scheme~$X$ of finite type over an algebraically closed field~$k$ it is not to be expected that Poincar\'e duality holds for all constructible sheaves.
We need to restrict to a certain kind of constructible sheaves that we will define in this section: the horizontally constructible sheaves.

We follow the terminology of \cite{EGAI} and \cite{stacks-project} concerning constructible subsets:
A subset~$T$ of a topological space~$X$ is called constructible
 if it is contained in the boolean algebra generated by the retrocompact open subsets.
It is called locally constructible if every point $x \in X$ has an open neighborhood $U \subseteq X$
 such that $T \cap U$ is constructible in~$U$.

We fix a noetherian ring~$\Lambda$ and a topology $\tau \in \{\et,t,\set\}$.
As for the \'etale site, constructible sheaves are defined as follows (see \cite{Hu96}, Definition~2.7.2)

\begin{definition}
 Let~$X$ be a pseudo-adic space (see \cite{Hu96}, Definition~1.10.3).
 A sheaf~$\mathcal{F}$ of $\Lambda$-modules on $X_{\tau}$ is constructible if for every point $x \in |X|$ there is a locally constructible subset $L \subseteq |X|$ containing~$x$ such that~$\mathcal{F}|_L$ is locally constant of finite type.
\end{definition}

Note that Huber's notion of a locally constructible subsets differs from ours.
However, given a point $x \in |X|$ as in the above definition
 and a locally constructible subset $L \subseteq |X|$ in Huber's sense
 (i.e.,~$L$ is locally constructible in our sense in an open $U$ of~$X$) containing~$x$,
 we obtain a locally constructible subset $T$ in our sense such that $x \in T \subset L$
 by intersecting~$L$ with any quasi-compact open neighborhood of~$x$ contained in~$U$.
Hence, in the definition of a constructible sheaf it does not matter whether we use our or Huber's definition of a locally constructible subset.

Horizontally constructible sheaves are special kinds of constructible sheaves.
We will only define them for discretely ringed adic spaces, not for general pseudo-adic spaces, as this is all we need in this article.
Note, however, that the definition of constructible sheaves involves pseudo-adic spaces even if we only want to define them on adic spaces as we need to make sense of a sheaf on a constructible subset of the adic space.
This is not the case for horizontally constructible sheaves.
Let us call a subset of an adic space~$X$ \emph{Zariski-constructible} if it is constructible with respect to the Zariski-topology.
This is the topology whose closed subsets are of the form
\[
 \{x \in X \mid |f(x)| = 0~\forall\,f \in \mathcal{I}_x\} = \{x \in X \mid \mathcal{I}_x \ne \cO_{X,x}\},
\]
 where~$\mathcal{I}$ is a coherent ideal of the structure sheaf of~$X$. 
In case $X = \Spa(Z,S)$ for a morphism of schemes $Z \to S$, this is the pullback of the Zariski topology on~$Z$
 via the support map $X \to Z$.

Note that in contrast to general locally closed constructible subsets the locally closed \emph{Zariski} constructible subsets carry the structure of an adic space.

\begin{definition}
 Let~$X$ be a discretely ringed adic space.
 A sheaf~$\mathcal{F}$ of $\Lambda$-modules on $X_{\tau}$ is horizontally constructible if for every point $x \in |X|$ there is a locally closed Zariski-constructible subset $L \subseteq |X|$ containing~$x$
  such that~$\mathcal{F}|_L$ is locally constant of finite type.
\end{definition}

The term horizontally constructible derives from the classification of specializations of points in an adic space.
Every specialization can be decomposed in horizontal and vertical specializations (primary and secondary in Huber's terminology).
The horizontally constructible sheaves~$\mathcal{F}$ have the property that for two geometric points~$\bar{x}$ and~$\bar{y}$ such that~$\bar{x}$ is a vertical specialization of~$\bar{y}$
 the specialization map $\mathcal{F}_{\bar{x}} \to \mathcal{F}_{\bar{y}}$ is an isomorphism.

We write $\Mod_c^h(X_{\tau},\Lambda)$ for the category of horizontally constructible $\Lambda$-modules on~$X_{\tau}$.
It is a thick subcategory of the category of $\Lambda$-modules on~$X_{\tau}$.
The pullback via an adic morphism $X \to Y$ takes horizontally constructible $\Lambda$-modules on~$Y_{\tau}$ to horizontally constructible $\Lambda$-modules on~$X_{\tau}$.

Let $f: X \to Y$ be a morphism in $Y_{\tau}$.
As for any site (see \cite{SGA4}, Expos\'ee~III, \S~5), we have the extension by zero functor
\[
 f_! : \Mod(X_{\tau},\Lambda) \longrightarrow \Mod(Y_{\tau},\Lambda)
\]
defined as the left adjoint of~$f^*$.
It is exact and takes a $\Lambda$-module~$\mathcal{F}$ on~$X_{\tau}$ to the sheaf associated with the presheaf
\[
 (U \to Y) \mapsto \bigoplus_{\varphi \in \Hom_Y(U,X)} \mathcal{F}(U \overset{\varphi}{\to} X).
\]
This construction is functorial in~$f$ and commutes with base change.
The stalk of $f_!\mathcal{F}$ at a geometric point $s:\bar{y} \to Y$ is isomorphic to
\[
 \bigoplus_{t \in \Hom_Y(\bar{y},X)} \mathcal{F}_t.
\]
If~$f$ is quasi-compact and quasi-separated, $f_!$ takes horizontally constructible sheaves to horizontally constructible sheaves.

\section{Cohomological dimension of horizontally constructible sheaves} \label{cohdim-sec}

\begin{definition}
 Let~$X$ be an adic space and~$p$ a prime.
 The (horizontal) cohomolo\-gical $p$-dimension $cd_p(X)$ ($\cd^h_p(X)$) of~$X$ is the maximal integer~$n$ such that there is a (horizontally) constructible $p$-torsion sheaf~$\mathcal{F}$ on~$X_{\tau}$ with $H^n(X,\mathcal{F}) \ne 0$.
\end{definition}

For the rest of this section we fix a separably closed field~$k$ of characteristic $p > 0$ and a curve~$C$ over~$k$, i.e., a one-dimensional scheme of finite type over~$k$.
We set $X = \Spa(C,k)$ and set out to determine the horizontal cohomological $p$-dimension of~$X$.
For our computation we will need to know beforehand that the horizontal cohomological $p$-dimension is finite.
As $\cd^h_p(X) \le \cd_p(X)$, this can be established by first determining the (general, not horizontal) cohomological $p$-dimension of~$X$.

\subsection{Cohomological dimension of proper closed subspaces}

At first we will need to compute  the cohomological dimension of proper closed subsets of~$X$.
Unfortunately, a closed subset of~$X$ is not necessarily an adic space but only a pseudo-adic space.
Remember that a prepseudo-adic space~$Y$ is a pair $(\underline{Y},|Y|)$, where~$\underline{Y}$ is an adic space and $|Y| \subset \underline{Y}$ is a subset.
A pseudo-adic space is a prepseudo-adic space $Y$ such that $|Y|$ is convex and locally pro-constructible in~$\underline{Y}$.
This is the case, for instance, if~$|Y|$ is closed in~$\underline{Y}$.
A morphism $f:Y \to Z$ of prepseudo-adic spaces is a morphism of adic spaces $\underline{Y} \to \underline{Z}$ (also denoted~$f$) such that $f(|Y|) \subseteq |Z|$.
It is called \'etale if the underlying morphism $\underline{Y} \to \underline{Z}$ is \'etale and tame (resp. strongly \'etale) if it is tame (resp. strongly \'etale) at all points $x \in |Y|$.
A family of \'etale (tame, strongly \'etale) morphisms $(Y_i \to Y)_{i \in I}$ of prepseudo-adic is a covering if $(|Y_i| \to |Y|)$ is a surjective family.

Let us get back to the proper closed subspaces of $X = \Spa(C,k)$.
We start by classifying the irreducible closed subsets.

\begin{lemma} \label{classification_irreducible_subspaces}
 There are four types of irreducible closed subsets of~$X$:
 \begin{enumerate}[(i)]
  \item $\{x\}$, where $x$ corresponds to the trivial valuation on the residue field of a closed point of~$C$,
  \item $\{v\}$, where $v$ corresponds to a $k$-valuation on a generic point of~$C$ which does not have center in~$C$,
  \item $\{v,x\}$, where $v$ corresponds to a $k$-valuation on a generic point of~$C$ with center $c \in C$ and~$x$ corresponds to the trivial valuation on~$k(c)$,
  \item $\Spa(Z,k)$, where $Z$ is an irreducible component of~$C$.
 \end{enumerate}
\end{lemma}

\begin{proof}
 Let~$Z \ne \varnothing$ be a closed irreducible subset of~$X$.
 Denote by~$z$ its generic point.
 If the support of~$z$ is a closed point of~$C$, the corresponding valuation has to be the trivial one and~$z$ is a closed point. We are then in case~(i).
 Suppose now that the support of~$z$ is a generic point of~$C$.
 Let us assume first that the corresponding valuation is not the trivial one.
 Hence, the valuation is discrete of rank one.
 If it does not have a center in~$C$, we are in case (ii).
 If it does have a center $c \in C$, the trivial valuation on~$c$ is contained in the closure of~$z$ and we are in case~(iii).
 Finally, if the valuation is the trivial one, we are in case (iv).
\end{proof}

Note that the underlying topological space of $X$ is noetherian.
This can be checked, for instance, by verifying that the irreducible closed subspaces listed above are constructible.
This is a very unusual property of an adic space and it is not true in higher dimension.

\begin{lemma} \label{classification_connected_subspaces}
 Let $Z$ be a connected closed subset of~$X$ whose complement is dense.
 Then it is either irreducible of type~(i) or~(ii) or of the form $\{v_1,\ldots,v_n,x\}$, where the~$v_i$ correspond to valuations on a generic point of~$C$ having all the same center~$c$
  and~$x$ corresponds to the trivial valuation on~$c$.
\end{lemma}

\begin{proof}
 The irreducible components of~$Z$ have to be of type~(i), (ii), or (iii).
 Subspaces of type~(i) and~(ii) consist of only one point.
 Hence, they cannot intersect non-trivially with other closed subspaces.
 If an irreducible component of~$Z$ is of type~(i) or~(ii), $Z$ has thus to be equal to it.
 It remains to treat the case where all irreducible components of~$Z$ are of type~(iii).
 As~$Z$ is connected, they have to intersect.
 It is then clear that~$Z$ has to be of the form $\{v_1,\ldots,v_n,x\}$ as described in the lemma.
\end{proof}

If $Z$ is of the form $\{v_1,\ldots,v_n,x\}$, we say (by abuse of notation) that it is of type~(iii).
In order to compute the cohomological dimension of the spaces of type~(i) and~(ii) the following lemma will be used.

\begin{lemma} \label{henselization}
 Let~$X$ be an adic space and $x \in X$ a point.
 Let $(k,k^+)$ be an affinoid field with henselization $(K,K^+)$.
 Suppose we are given an isomorphism $ \varphi: \Spa(k,k^+) \to \Spa(\kappa(x),\kappa(x)^+)$.
 Denote by $K^{\tau}$ a maximal separable, tamely ramified, or unramified extension of~$K$ (with respect to $K^+$) according to whether~$\tau$ is the \'etale, the tame or the strongly \'etale topology.
 Then~$\varphi$ induces an equivalence of topoi
 \[
  \Gal(K^{\tau}|K)\text{-Sets} \to \Shv(X,\{x\})_{\tau}.
 \]
\end{lemma}
\begin{proof}
 Denote by~$\bar{x}$ the closed point of $\Spa(K,K^+)$
 We have the following chain of morphisms of pseudo-adic spaces:
 \[
  (\Spa(K,K^+),\bar{x}) \to (\Spa(k,k^+),x) \overset{\varphi}{\to} (\Spa(\kappa(x),\kappa(x)^+),x) \to (X,x)
 \]
 In \cite{Hu96}, Proposition 2.3.10 it is proved that the associated morphism of \'etale topoi is an equivalence.
 The proof for the tame and strongly \'etale topoi is literally the same.
 Moreover, the standard argument (see \cite{SGA4}, Expos\'e~viii, \S 2) shows that we have an equivalence of topoi
 \[
  \Gal(K^{\tau}|K)\text{-Sets} \to \Shv(\Spa(K,K^+),\bar{x})_{\tau}. \qedhere
 \]
\end{proof}

%



\begin{proposition} \label{cd_proper_subspace}
 The cohomological $p$-dimension with respect to $\tau \in \{\set,t\}$ of a connected closed subspace of~$X$ is~$0$ if it is of type~(i) or~(ii), and less or equal to~$1$ if it is of type~(iii).
\end{proposition}

\begin{proof}
 Let $\{x\}$ be a closed subspace of type~(i).
 The residue field $k(x)$ is separably closed.
 Hence, by \cref{henselization} the topos $\Shv(X,x)_{\tau}$ is trivial.
 For a closed subspace $\{v\}$ of type~(ii) let $(K,K^+)$ denote the henselization of $(\kappa(x),\kappa(x)^+)$.
 The residue field of~$K^+$ is the same as the one of~$\kappa(x)^+$, which is separably closed.
 Thus, $(K,K^+)$ is even strictly henselian.
 Using \cref{henselization} we conclude that $\Shv(X,v)_{\set}$ is trivial.
 Moreover, the cohomological $p$-dimension of $(X,v)$ in the tame topology is zero because $\Gal(K^t|K)$ has trivial $p$-Sylow subgroup.

 Now take a closed subspace $\{v_1,\ldots,v_n,x\}$ of type~(iii) and an abelian $p$-torsion sheaf~$\mathcal{F}$ on $(X,\{v_1,\ldots,v_n,x\})_{\tau}$.
 Consider the open subspace $(U,\{v_i,\ldots,v_n\})$, where
 \[
  U = X-\{x\} = \Spa(C-\{\supp x\},k).
 \]
 and denote by
 \[
  j : (U,\{v_i,\ldots,v_n\}) \to (X,\{v_1,\ldots,v_n,x\})
 \]
 the canonical embedding.
 We obtain an exact sequence
 \[
  0 \longrightarrow \mathcal{G} \longrightarrow \mathcal{F} \longrightarrow j_* j^* \mathcal{F} \longrightarrow \mathcal{H} \longrightarrow 0.
 \]
 The sheaves~$\mathcal{G}$ and~$\mathcal{H}$ have support in $(X,x)$, hence trivial cohomology.
 In order to compute the cohomology of $j_*j^* \mathcal{F}$, consider the Leray spectral sequence associated to~$j$:
 \[
  H^r((X,\{v_1,\ldots,v_n,x\}),R^sj_*(j^*\mathcal{F})) \Rightarrow H^{r+s}((U,\{v_i,\ldots,v_n\}),j_*j^*\mathcal{F}).
 \]
 We have seen above that the cohomological dimension of $(U,\{v_1,\ldots,v_n\})$ is~$0$.
 Therefore, the right hand side vanishes unless $r=s=0$ and $R^sj_*(j^*\mathcal{F}) = 0$ for $s \ge 1$.
 This shows that
 \[
  H^r((X,\{v_1,\ldots,v_n,x\}),j_*j^*\mathcal{F}) = 0
 \]
 for $r \ge 1$.

 We split the above sequence into two short exact sequences
 \[
  0 \longrightarrow \mathcal{G} \longrightarrow \mathcal{F} \longrightarrow \mathcal{J} \longrightarrow 0
 \]
 and
 \[
  0 \longrightarrow \mathcal{J} \longrightarrow j_*j^*\mathcal{F} \longrightarrow \mathcal{H} \longrightarrow 0.
 \]
 Using that $\mathcal{G}$, $\mathcal{H}$, and $j_*j^*\mathcal{F}$ are cohomologically trivial, the long exact cohomology sequences shows that
 \[
  H^r((X,\{v_1,\ldots,v_n,x\}),\mathcal{F}) = 0
 \]
 for $r \ge 2$.
\end{proof}

\subsection{Cohomological dimension of curves}

In this subsection we give a bound for the cohomological $p$-dimension of $X = \Spa(C,k)$.
The strategy of the proof is the same as for the \'etale cohomological $p$-dimension of schemes in case $p$ is invertible (see \cite{SGA4}, Expos\'{e}~X, \S~4).

\begin{proposition} \label{cd_constructible}
 The cohomological $p$-dimension of~$X$ is less or equal to~$3$.
\end{proposition}

\begin{proof}
 Take a constructible $p$-torsion sheaf~$\mathcal{F}$ on~$X_{\tau}$.
 Let $A \hookrightarrow C$ be the inclusion of the generic points of~$C$.
 It induces a morphism $\iota: \Spa(A,A) \to X$.
 We consider the adjunction homomorphism $\mathcal{F} \to \iota_*\iota^*\mathcal{F}$ and define~$\mathcal{G}$ and~$\mathcal{H}$ to be its kernel and cokernel.
 In other words, we have an exact sequence
 \begin{equation} \label{adjunction_kernel_cokernel}
  0 \to \mathcal{G} \to \mathcal{F} \to \iota_*\iota^*\mathcal{F} \to \mathcal{H} \to 0.
 \end{equation}
 The stalk of the adjunction morphism in the middle is an isomorphism for any geometric point lying over the generic points $\Spa(A,A)$ of~$X$.
 Hence, $\mathcal{G}$ and~$\mathcal{H}$ have support in a proper closed subspace and thus their cohomology vanishes in degrees $\ge 2$ (see \cref{cd_proper_subspace}).

 In order to examine the cohomology of $\iota_*\iota^*\mathcal{F}$ we consider the Leray spectral sequence associated with~$\iota$:
 \[
  H^r(X,R^s\iota_*\iota^*\mathcal{F}) \Rightarrow H^{r+s}(\Spa(A,A),\iota^*\mathcal{F}).
 \]
 The adic space $\Spa(A,A)$ is a finite disjoint union of spaces of the form $\Spa(K,K)$, where~$K$ is a field of characteristic $p$.
 The $\tau$-cohomology of these spaces is (independently of $\tau\in\{\et, t, \set\}$) given as the group cohomology of the absolute Galois group of $K$ (see \cite{HueAd}, Lemma~8.1).
 The Galois cohomology of a field of characteristic~$p$ with $p$-torsion coefficients vanishes in degrees greater or equal to~$2$ (see \cite{NSW}, Proposition~6.5.10).
 Therefore, the right hand side of the above spectral sequence vanishes for $r+s \ge 2$.
 By the same reason $R^s\iota_*\iota^*\mathcal{F} = 0$ for $s \ge 2$.
 Moreover, $R^1\iota_*\iota^*\mathcal{F}$ can be written as a filtered colimit of constructible sheaves with support in a proper closed subspace.
 As~$X_{\tau}$ is a noetherian site ($X$ being quasi-compact), cohomology commutes with colimits (see \cite{SGA4}, Expos\'ee~VI, Corollaire~5.3).
 Therefore, the left hand side vanishes for $s=1$ and $r \ge 2$.
 With this information, we obtain that
 \[
  H^r(X,\iota_*\iota^*\mathcal{F}) = 0
 \]
 for $r \ge  4$.

 Let us go back to sequence (\ref{adjunction_kernel_cokernel}).
 We know that the cohomology of~$\mathcal{G}$ and~$\mathcal{H}$ vanishes in degrees greater or equal to~$2$ and we just saw that the cohomology of~$\iota_*\iota^*\mathcal{F}$ vanishes in degrees greater or equal to~$4$.
 Splitting the sequence into two short exact sequences and examining the corresponding long exact cohomology sequences, we conclude that $H^r(X,\mathcal{F}) = 0$ for $r \ge 4$.
\end{proof}

\subsection{Computation of the horizontal $p$-dimension}

In the following, by a $\tau$-covering we mean a finite morphism of adic spaces $Y' \to Y$ in $Y_{\tau}$.

\begin{lemma} \label{cohdim_loc_const}
 Let~$Y$ be a noetherian adic space of finite horizontal $p$-dimension for some prime~$p$.
 For an integer~$n$ the following are equivalent:
 \begin{enumerate}[(i)]
  \item For every locally constant constructible $p$-torsion sheaf~$\mathcal{F}$ and every $i >n$ we have
	\[
	 H^i(Y,\mathcal{F}) = 0.
	\]
  \item For every $\tau$-covering $Y' \to Y$ and every $i > n$ we have
	\[
	 H^i(Y',\F_p) = 0.
	\]
 \end{enumerate}
\end{lemma}

\begin{proof}
 Without loss of generality we may assume that~$Y$ is connected.
 Assume that~(i) holds and take a $\tau$-covering $\pi:Y' \to Y$ of degree~$d$. Then for every $i$ we have
 \[
  H^i(Y',\F_p) = H^i(Y,\pi_*\F_p).
 \]
 But $\pi^*\pi_*\F_p \cong (\F_p)^d$ is constant, hence $\pi_*(\Z/ p\Z)$ is locally constant constructible.
 Therefore, by assumption the above cohomology group vanishes for $i > n$.

 Now assume~(ii).
 First observe that by d\'evissage we know that for every $\tau$-covering
 \[
  H^i(Y',M) = 0
 \]
 for all constant $p$-torsion sheaves~$M$ and all $i > n$.
 Let~$\mathcal{F}$ be a locally constant constructible $p$-torsion sheaf on~$Y_{\tau}$ and $\pi:Y' \to Y$ a $\tau$-covering such that $\pi^*\mathcal{F}$ is constant.
 The adjunction map $\pi_!\pi^*\mathcal{F} \to \mathcal{F}$ is surjective, which can be checked on stalks using that~$\pi$ is surjective.
 As~$\pi$ is finite, $\pi_!\pi^*\mathcal{F} = \pi_*\pi^*\mathcal{F}$, which is locally constant by the same argument as before.
 We thus have a short exact sequence of locally constant constructible $p$-torsion sheaves
 \[
  0 \longrightarrow \mathcal{G} \longrightarrow \pi_*\pi^*\mathcal{F} \longrightarrow \mathcal{F} \longrightarrow 0.
 \]
 Using that $H^i(Y,\pi_*\pi^*\mathcal{F}) = H^i(Y',\pi^*\mathcal{F}) = 0$ for $i>n$ we obtain
 \[
  H^i(Y,\mathcal{F}) \overset{\sim}{\longrightarrow} H^{i+1}(Y,\mathcal{G})
 \]
 for $i>n$.
 In particular, $H^i(Y,\mathcal{F}) = 0$ for $n < i \le \cd_p^h(Y)$.
 As this holds for all locally constant horizontally constructible $p$-torsion sheaves, we also have $H^i(Y,\mathcal{G}) = 0$ for $n < i \le \cd_p^h(Y)$.
 By descending induction we obtain the result.
\end{proof}

We can now use \cref{cohdim_loc_const} to determine the horizontal  cohomological dimension of $X = \Spa(C,k)$.

\begin{lemma} \label{coh_dim_loc_const_curves}
 For $\tau \in \{t,\set\}$ and every locally constant constructible $p$-torsion sheaf~$\mathcal{F}$ on~$X_{\tau}$ we have
 \[
  H^i(X_{\tau},\mathcal{F}) = 0
 \]
 for $i \ge 2$.
\end{lemma}

\begin{proof}
 We verify condition~(ii) in \cref{cohdim_loc_const}.
 Let $X' \to X$ be a $\tau$-covering.
 It comes from an \'etale covering $C' \to C$, where~$C'$ is a smooth curve.
 We denote by~$\bar{C}'$ the smooth compactification of~$C'$ and set $\bar{X}'=\Spa(\bar{C}',k)$.
 Writing~$Z'$ for the complement of~$X'$ in~$\bar{X}'$ we consider the excision sequence
 \[
  \ldots \to H^i_{Z'}(\bar{X}'_{\tau},\F_p) \to H^i(\bar{X}'_{\tau},\F_p) \to H^i(X'_{\tau},\F_p) \to \ldots
 \]
 By purity (\cref{purity}), the cohomology groups with support $H^i_{Z'}(\bar{X}',\F_p)$ vanish.
 Moreover, by \cite{HueAd}, Corollary~8.7
 \[
  H^i(\bar{X}'_{\tau},\F_p) \cong H^i(\bar{C}'_{\et},\F_p)
 \]
 for all $i \in \Z$ and the latter cohomology groups vanish for $i \ge 2$ by \cite{SGA4}, Expos\'e~X, Corollaire~5.2.
 Combining these information we obtain the result.
\end{proof}

\begin{proposition} \label{cd}
 We have $\cd^h_p(X) \le 1$ for $\tau \in \{t,\set\}$.
\end{proposition}

\begin{proof}
 Take a  horizontally constructible $p$-torsion sheaf~$\mathcal{F}$ on~$X_{\tau}$.
 Let $U \subseteq X$ be a Zariski-open subspace such that $\mathcal{F}|_U$ is locally constant.
 Denote by~$Z$ the complement of~$U$ and consider the corresponding excision sequence
 \[
  \ldots \longrightarrow H^i_Z(X_{\tau},\mathcal{F}) \longrightarrow H^i(X_{\tau},\mathcal{F}) \longrightarrow H^i(U_{\tau},\mathcal{F}) \longrightarrow \ldots
 \]
 By \cref{cohdimsupp}, the cohomology groups with support in~$Z$ vanish for $i \ge 2$.
 Moreover, by \cref{coh_dim_loc_const_curves}, we have $H^i(U_{\tau},\mathcal{F}) = 0$ for $i \ge 2$.
 This proves that $\cd^h_p(X) \le 1$.
\end{proof}

\begin{proposition} \label{cd_compact_support}
 For $\tau \in \{t,\set,\et\}$ and any horizontally constructible $p$-torsion sheaf~$\mathcal{F}$ on~$X_{\tau}$ we have
 \[
  H^i_c(X,\mathcal{F}) = 0
 \]
 for any $i\ge 2$.
\end{proposition}

\begin{proof}
 Let $\bar{C}$ be a compactification of~$C$ and set $\bar{X} = \Spa(\bar{C},k)$.
 Denote by $j$ the inclusion $X \hookrightarrow \bar{X}$.
 Writing $c: \bar{X} \to \bar{C}$ for the center map and using \cite{HueAd}, Corollary~8.7, we obtain
 \[
  H^i_c(X_{\tau},\mathcal{F}) = H^i(\bar{X}_{\tau},j_!\mathcal{F}) = H^i(\bar{C}_{\et},c_*j_!\mathcal{F}),
 \]
 which vanishes for $i \ge 2$ since the \'{e}tale $p$-cohomological dimension of $\bar{C}$ is one by \cite{SGA4}, Expos\'e~X, Corollaire~5.2.
\end{proof}

Let~$A$ be a noetherian ring and denote by~$A$ also the constant sheaf with stalks~$A$ on the~$\tau$-site of an adic space~$X$.
Remember that a sheaf~$\mathcal{F}$ of $A$-modules on~$X_{\tau}$ is called \emph{pseudocoherent} if locally on~$X_{\tau}$ there is an exact sequence of the form
\[
 A^m \to A^n \to \mathcal{F} \to 0
\]
with nonnegative integers~$m$ and~$n$.

\begin{lemma} \label{pseudocoherent}
 Let~$\mathcal{F}$ and~$\mathcal{G}$ be sheaves of $A$-modules with $\mathcal{F}$ pseudocoherent.
 Then for any geometric point $\bar{x}$ of~$X_{\tau}$ and any $n \ge 0$ we have
 \[
  \underline{\Ext}_A^n(\mathcal{F},\mathcal{G})_{\bar{x}} = \Ext_A^n(\mathcal{F}_{\bar{x}},\mathcal{G}_{\bar{x}}).
 \]
\end{lemma}

\begin{proof}
 We have
 \[
  \underline{\Hom}_A(\mathcal{F},\mathcal{G})_{\bar{x}} = \Hom_A(\mathcal{F}_{\bar{x}},\mathcal{G}_{\bar{x}})
 \]
 if~$\mathcal{F}$ is a finite free $A$-module.
 Using the five lemma we extend this assertion to pseudocoherent sheaves.
 The lemma follows from this by using the above identification for an injective resolution of~$\mathcal{G}$ once we have shown that the stalks of an injective sheaf~$\mathcal{H}$ of $A$-modules are injective $A$-modules.
 In order to see this it is enough to consider an $A$-module homomorphism
 \[
  I \to \mathcal{H}_{\bar{x}}
 \]
 from an ideal~$I$ of~$A$ to a stalk of~$\mathcal{H}$ and show that it prolongs to~$A$.
 As~$I$ is finitely generated ($A$ being noetherian), it comes from an $A$-homomorphism $I \to \mathcal{H}(U)$ for a neighborhood~$j:U \to X$ of~$\bar{x}$.
 This induces a homomorphism from the constant sheaf~$I$ on~$U$ to~$j^*\mathcal{H}$.
 But~$j^*$ has an exact left adjoint (namely~$j_!$) and thus takes injectives to injectives.
 Hence, we can prolong the homomorphism to a homomorphism $A \to j^*\mathcal{H}$.
 In particular, this prolongs our inicial homomorphism $I \to \mathcal{H}_{\bar{x}}$ to~$A$.
\end{proof}

\begin{proposition} \label{cd_ext}
 For $\tau \in \{t,\set\}$ and any horizontally constructible $\F_p$-module~$\mathcal{F}$ on~$X_{\tau}$ we have
 \[
  \Ext_{\F_p}^i(\mathcal{F},\nu_X(1)) = 0
 \]
 for all $i \ge 2$.
\end{proposition}

\begin{proof}
 If~$\mathcal{F}$ is locally constant, we consider the local-to-global spectral sequence
 \[
  H^i(X_{\tau},\underline{\Ext}_{\F_p}^j(\mathcal{F},\nu_X(1)) \Rightarrow \Ext_{\F_p}^{i+j}(\mathcal{F},\nu_X(1)).
 \]
 The sheaf~$\mathcal{F}$ is a pseudocoherent $\F_p$-module.
 Hence, we can use \cref{pseudocoherent} to compute the stalks of the $\Ext$-sheaf at a geometric point~$\bar{x}$:
 \[
  \underline{\Ext}_{\F_p}^j(\mathcal{F},\nu_X(1))_{\bar{x}} = \Ext_{\F_p}^j(\mathcal{F}_{\bar{x}},\nu_{X_{\bar{x}}}(1)).
 \]
 This vanishes for $j \ge 1$ as~$\mathcal{F}_{\bar{x}}$ is a projective $\F_p$-module.
 Therefore, the spectral sequence degenerates and we obtain for all $i \ge 0$:
 \[
  \Ext_{\F_p}^i(\mathcal{F},\nu_X(1)) \cong H^i(X_{\tau},\underline{\Hom}_{\F_p}(\mathcal{F},\nu_X(1))).
 \]
 But $\underline{\Hom}_{\F_p}(\mathcal{F},\nu_X(1))$ is a horizontal $p$-torsion sheaf and thus the result follows from \cref{cd}.

 In the general case we find a Zariski-open $j : U \hookrightarrow X$ such that $j^*\mathcal{F}$ is locally constant.
 Denote by $i : Z \hookrightarrow X$ the complement of~$U$ and consider the short exact sequence
 \[
  0 \longrightarrow j_!j^*\mathcal{F} \longrightarrow \mathcal{F} \longrightarrow i_*i^*\mathcal{F} \longrightarrow 0.
 \]
 Applying $\Ext_{\F_p}^n(-,\nu_X(1))$ and using $i^!\nu_X(1) \cong \F_p[-1]$ (see \cref{purity_sheaf}) we obtain a long exact sequence
 \[
  \ldots \to \Ext_{\F_p}^{n-1}(i^*\mathcal{F},\F_p) \to \Ext_{\F_p}^n(\mathcal{F},\nu_X(1)) \to \Ext_{\F_p}^n(j^*\mathcal{F},\nu_U(1)) \to \ldots
 \]
 The left hand group vanishes for $n \ge 2$ as~$Z$ consists of finitely many points of the form $\Spa(k,k)$ and~$k$ is separably closed.
 The right hand group also vanishes for $n \ge 2$ by the case of locally constant sheaves we have seen above.
 Hence,
 \[
  \Ext_{\F_p}^n(\mathcal{F},\nu_X(1)) = 0
 \]
 for $n \ge 2$.
\end{proof}

\section{Finite morphisms}

\begin{proposition} \label{basechange}
 Consider a Cartesian diagram
 \[
  \begin{tikzcd}
   X'	\ar[r,"g'"]	\ar[d,"f'"']	& X	\ar[d,"f"]	\\
   S'	\ar[r,"g"]			& S
  \end{tikzcd}
 \]
 of adic spaces, where~$f$ is finite and~$g$ is a locally closed immersion.
 Then for any sheaf~$\mathcal{F}$ on $X_{\tau}$ the base change morphism
 \[
  \varphi:g^*f_*\mathcal{F} \to f'_*g'^*\mathcal{F}
 \]
 is an isomorphism.
\end{proposition}

\begin{proof}
 We can check this on stalks.
 Let $\bar{s}' \to S'$ be a geometric point (with respect to~$\tau$).
 Replacing~$S$ and~$S'$ by their strict localizations at~$g(\bar{s}')$ and $\bar{s}'$, we may assume that~$S$ and~$S'$ are strictly local with closed point $g(\bar{s}')$ and~$\bar{s}'$, respectively.
 This also reduces us to the case where~$g$ is a closed immersion.
 As~$f$ is finite, $X$ is a finite disjoint union of strictly local adic spaces.
 By treating each of them separately we reduce to the case where~$X$ is strictly local.
 Then~$X'$ is also strictly local being a closed subspace of~$X$.
 The stalk of~$\varphi$ at~$\bar{s}'$ is the map
 \[
  \varphi_{\bar{s}'} : \mathcal{F}(X) \to g'^*\mathcal{F}(X') = \mathcal{F}(X^{\tau}_{g(\bar{s}')}) = \mathcal{F}(X),
 \]
 which is obviously an isomorphism.
\end{proof}

The following corollary also holds in greater generality but in order to keep things simple we only prove the following special case:

\begin{corollary} \label{pushforward_constructible}
 Let~$C$ be a curve over a field~$k$ and $C' \to C$ a finite flat morphism which is generically \'etale.
 Set $X = \Spa(C,k)$, $X'= \Spa(C',k)$, and consider the induced morphism $f: X' \to X$.
 Assume that there is a dense Zariski-open~$V$ of~$X$ such that $f|_V$ is a finite $\tau$-morphism.
 Then for every horizontally constructible sheaf~$\mathcal{F}$ on~$X'_{\tau}$ the pushforward $f_*\mathcal{F}$ is horizontally constructible.
\end{corollary}

\begin{proof}
 Let us show first that there is a dense Zariski-open subspace $U$ of~$X$ such that the restriction of~$f_*\mathcal{F}$ to~$U$ is locally constant constructible.
 By assumption there is a dense Zariski-open subspace~$V$ of~$X$ such that
 \[
  f_V : V' := V \times_X X' \to V
 \]
 is a finite $\tau$-morphism.
 Then $(f_V)_!$ is defined and takes horizontally constructible sheaves to horizontally constructible sheaves.
 Since~$f$ is finite, $(f|_V)_!$ coincides with $(f|_V)_*$.
 Therefore, on a Zariski-open~$U$ of~$V$ we have that $(f_U)_*\mathcal{F}|_U$ is locally constant constructible and we are done.

 It remains to show that the stalks of $f_*\mathcal{F}$ at geometric points~$\bar{x}$ of~$X$ with image in the complement of~$U$ are finite.
 These can be taken of the form $\bar{x} = \Spa(\bar{k},\bar{k}) \to X$ for a separable closure~$\bar{k}$ of the base field~$k$.
 Consider the Cartesian diagram
 \[
  \begin{tikzcd}
   X' \times_X \bar{x}	\ar[r,"g'"]	\ar[d,"f'"']	& X' \times_X X_{\bar{x}}	\ar[d,"f"]	\\
   \bar{x}		\ar[r,"g"]			& X_{\bar{x}}.
  \end{tikzcd}
 \]
 It is of the form as in \cref{basechange}.
 Hence, the corresponding base change morphism~$\varphi$ is an isomorphism.
 In our situation $X' \times_X \bar{x}$ is a finite product of spaces of the form $\Spa(\bar{k},\bar{k})$ indexed by the liftings of~$\bar{x}$ to~$X'$.
 Thus~$\varphi$ identifies with the natural map
 \[
  (f_*\mathcal{F})_{\bar{x}} \to \prod_{\bar{x}' \to X'} \mathcal{F}_{\bar{x}'},
 \]
 where the product runs over all geometric points~$\bar{x}'$ lying over~$\bar{x}$.
 In particular, we see that the stalk $(f_*\mathcal{F})_{\bar{x}}$ is finite.
\end{proof}

\section{Poincar\'e duality} \label{duality-sec}

\subsection{Construction of the pairing} \label{construction_pairing}

Let~$f: C \to \Spec k$ be a smooth connected curve over a perfect field~$k$ and set $X = \Spa(C,k)$.
We write $\pi$ for the natural morphism $\Spa(f):X \to \Spa(k,k)$.
Denote by $\bar{f}:\bar{C} \to \Spec k$ the smooth compactification of~$C$ over~$k$.
Then
\[
 \bar{\pi}:=\Spa(\bar{f}):\bar{X}:=\Spa(\bar{C},k) \to \Spa(k,k)
\]
is proper and~$X$ is an open subspace of~$\bar{X}$.
We denote the natural map $X \to \bar{X}$ by~$j$.

From \cite{Milne86}, Theorem 2.4 we have a functorial trace map
\[
 \tr: R\bar{f}_*\nu_{\bar{C}}(1) \to \F_p.
\]
It can be identified with a trace map (see \cite{HueAd}, Corollary~8.3)
\[
\tr: R\bar{\pi}^{\et}_*\nu_{\bar{X}}(1) \to \F_p.
\]
The \'etale, tame, and strongly \'etale sites coincide on $\Spa(k,k)$.
Moreover, $R\bar{\pi}^{\et}_*$, $R\bar{\pi}^{t}_*$, and $R\bar{\pi}^{\set}_*$ are naturally equivalent functors because~$\bar{f}$ is proper (see \cite{HueAd}, Corollary~8.7).
This gives a trace map
\[
 \tr: R\bar{\pi}^{\tau}_*\nu_{\bar{X}}(1) \to \F_p
\]
for any of the topologies $\tau \in\{\set,t,\et\}$.
In the following we just write $R\bar{\pi}_*$ instead of $R\bar{\pi}^{\tau}_*$.

For every horizontally constructible sheaf~$\mathcal{F}$ on~$X_{\tau}$ we consider the map
\begin{align} \label{definition_alpha}
 \alpha_X(\mathcal{F}): R\Hom_X(\mathcal{F},\nu_X(1))[1] &\cong R\Hom_{\bar{X}}(j_!\mathcal{F},\nu_{\bar{X}}(1))[1] \\
 &\to R\Hom_k(R\bar{\pi}_*j_!\mathcal{F},R\bar{\pi}_*\nu_{\bar{X}}(1))[1] \\
 &\overset{\tr}{\to} R\Hom_k(R\pi_!\mathcal{F},\F_p).
\end{align}
in $D^+(\F_p)$.

\subsection{Comparison with Milne's duality}

Let us compare $\alpha_X(\mathcal{F})$ with the duality of logarithmic de Rham-Witt sheaves proved by Milne in \cite{Milne86}.
We need to introduce some notation (see \cite{Milne76}, section~2).
Denote by $(\textit{Pf}/k)_{\et}$ the \'etale site on the category of perfect affine schemes over~$k$.
Moreover, let $\mathcal{G}(p^{\infty})$ denote the category of commutative algebraic perfect group schemes over~$k$ that are annihilated by a power of~$p$.
This is an abelian subcategory of $\Shv(\textit{Pf}/k)_{\et}$ which is equivalent to the category of unipotent quasi-algebraic group schemes studied in \cite{Serre60} (see \cite{Milne76}, Remark~2.3).
Every unipotent group scheme~$G$ fits into a short exact sequence
\[
 0 \to G^0 \to G \to G^{\et} \to 0,
\]
where~$G^0$ is unipotent and connected (it is the connected component of the identity of~$G$) and~$G^{\et}$ is \'etale.
\'{E}tale locally every unipotent connected algebraic group scheme has a composition series whose composition factors are isomorphic to $\GG_a$.
In the language of commutative algebraic perfect group schemes this gives for every connected object \'etale locally a composition series with composition factors isomorphic to $\GG_a^{\textit{Pf}}$, the perfection of~$\GG_a$.
In particular, the group of sections of a connected algebraic perfect group scheme over an algebraically closed field is always infinite.

For a smooth projective scheme~$Y$ of dimension~$d$ over a perfect field~$k$ Milne shows that the Yoneda pairing and the above described trace map induce isomorphisms
\begin{equation} \label{duality_Milne}
 \underline{H}^{\bullet}(Y,\nu(r)) \to \underline{H}^{\bullet}(Y,\nu(d-r))^t[-d]
\end{equation}
for all integers~$r$ (see \cite{Milne86}, Theorem~1.11).
Here, $\underline{H}^{\bullet}(Y,\nu(r))$ is the object of $\mathcal{G}(p^{\infty})$ representing the sheaf on $(\textit{Pf}/k)_{\et}$ associated with the presheaf
\[
 T \to H^i(Y_T,\nu(r))
\]
(see \cite{Milne86}, Lemma~1.2).
Furthermore, $(-)^t$ is the duality on $D^b(\mathcal{G}(p^{\infty})$ defined by
\[
 G^{\bullet t} = R\Hom(G^{\bullet},\Q_p/\Z_p)
\]
(see \cite{Milne86}, Lemma~1.3).

Taking the \'etale and the connected part of the group schemes in \cref{duality_Milne} we obtain isomorphisms
\begin{align*}
 U^i(Y,\nu(r)) &\cong U^{d+1-i}(Y,\nu(d-r))^v,	\\
 D^i(Y,\nu(r)) &\cong D^{d-i}(Y,\nu(d-r))^*,
\end{align*}
where
\begin{align*}
 U^i(Y,\nu(r)) &= \underline{H}^i(Y,\nu(r))^0,			\\
 D^i(Y,\nu(r)) &= \underline{H}^i(Y,\nu(r))/U^i(Y,\nu(r)),	\\
 (-)^* &= \Hom(-,\Q_p/\Z_p),	 \quad \text{and}		\\
 (-)^v &= \Ext^1(-,\Q_p/\Z_p).
\end{align*}
Let us examine the case $r=0$, i.e., $\nu(r) = \F_p$.
For every algebraically closed field~$K$ over~$k$ the cohomology groups $H^i(Y_K,\F_p)$ are finite (\cite{SGA4}, Expos\'e~XIV, Corollaire~1.2).
Hence, the connected part of $\underline{H}^i(Y,\F_p)$ has to be trivial.
We therefore obtain a duality
\[
 \underline{H}^i(Y,\F_p) \cong \underline{H}^{d-i}(Y,\nu(d))^*
\]
of \'etale group schemes.
If~$k$ is algebraically closed, taking global sections gives a duality
\begin{equation} \label{duality_Milne_global_sections}
 H^i(Y,\F_p) \cong H^{d-i}(Y,\nu(d))^*.
\end{equation}

\begin{proposition} \label{alpha=Milne}
 Suppose that~$k$ is algebraically closed and $Y = C$ is a smooth projective curve.
 The above isomorphism (\ref{duality_Milne_global_sections}) coincides with $\alpha_X(\F_p)$.
\end{proposition}

\begin{proof}
 Both maps are the composite of the Yoneda pairing with a trace map.
 By construction the two trace maps are the same.
\end{proof}

\subsection{Proof of Poincar\'e duality}

In this section we show that the homomorphism $\alpha_X(\mathcal{F})$ defined in \cref{construction_pairing} is a quasi-isomorphism for $\tau \in \{t,\set\}$.
In other words, Poincar\'e duality holds for the tame and the strongly \'etale cohomology.
The structure of the proof is taken from \cite{Geisser10}, section~4.
It is possible to transfer the structure of the proof in loc.\ cit.\ to our situation precisely because we have cohomological purity at our disposal.
Correspondingly it only works for the tame and the strongly \'etale topologies and not for the \'etale one.

\begin{lemma} \label{reduction_alg_closed}
 Suppose that~$\alpha_X(\mathcal{F})$ is a quasi-isomorphism whenever~$k$ is algebraically closed.
 Then~$\alpha_X(\mathcal{F})$ is a quasi-isomorphism for all perfect fields~$k$.
\end{lemma}

\begin{proof}
 Let~$\bar{k}$ be an algebraic closure of~$k$ and denote by~$G$ the Galois group of $\bar{k}|k$.
 We have a Grothendieck spectral sequence associated with the derived functors of $\mathit{Hom}_{\bar{X}}(\mathcal{F},-)$ and $(-)^G$ because $\mathit{Hom}_{\bar{X}}(\mathcal{F},-)$ maps injective sheaves to flabby ones (\cite{Milne80}, III, Corollary~2.13c).
 So
 \[
  R\Hom_X(\mathcal{F},\nu(d)) \cong R\Gamma_GR\mathit{Hom}_{\bar{X}}(\mathcal{F},\nu(d)),
 \]
 where~$\Gamma_G$ denotes the derived functor of $(-)^G$.
 Similarly
 \[
  R\Hom_k(Rf_!\mathcal{F},\F_p) \cong R\Gamma_G R\mathit{Hom}_{\bar{k}}(R\bar{f}_!\mathcal{F},\F_p) = R\Gamma_G \mathit{Hom}_{\bar{k}}(H^i_c(X_{\bar{k}},\mathcal{F}),\F_p).
 \]
 Checking that~$\alpha_X(\mathcal{F})$ and~$\alpha_{X_{\bar{k}}}(\mathcal{F})$ are compatible with these identifications, the lemma follows.
\end{proof}

\begin{lemma} \label{duality_etale_morphism}
 Let $C' \to C$ be an \'etale morphism inducing a $\tau$-morphism
 \[
  g: X' := \Spa(C',k) \to X.
 \]
 For every horizontally constructible sheaf $\mathcal{F}'$ of $\F_p$-modules on $X'_{\tau}$ the map $\alpha_{X'}(\mathcal{F}')$ is a quasi-isomorphism if and only if $\alpha_X(g_!\mathcal{F}')$ is.
\end{lemma}

\begin{proof}
 Denote by $\pi': X' \to \Spa(k,k)$ the structure morphism of~$X'$.
 Then
 \[
  R\pi_!(g_!\mathcal{F}') = R\pi'_!\mathcal{F}'.
 \]
 Moreover, using that $g^*\nu_X(1) = \nu_{X'}(1)$ and that $g_!$ is left adjoint to~$g^*$ we get a natural identification
 \[
  R\Hom_X(g_!\mathcal{F}',\nu_X(1)) = R\Hom_{X'}(\mathcal{F}',\nu_{X'}(1)).
 \]
 In this way we identify $\alpha_{X'}(\mathcal{F})$ with $\alpha_X(g_!\mathcal{F}')$.
\end{proof}

The following lemma only holds for the tame and the strongly \'etale topologies because it makes use of cohomological purity.

\begin{lemma} \label{duality_proper_closed_subspace}
 Assume that~$k$ is algebraically closed.
 Let $i : Z \hookrightarrow X$ be a proper Zariski-closed subspace (i.e., $Z = \Spa(S,k)$ for some proper closed subscheme~$S$ of~$C$).
 Then for $\tau \in \{t,\set\}$ and every abelian sheaf $\mathcal{F}$ on $Z_{\tau}$ the map $\alpha_X(i_*\mathcal{F})$ is a quasi-isomorphism.
\end{lemma}

\begin{proof}
 Without loss of generality we may assume that~$Z$ consists of a single point.
 Then $Z \cong \Spa(k,k)$ and $Z_{\tau}$ is trivial.
 Consider the diagram
 \[
  \begin{tikzcd}
   R\Hom_X(i_*\mathcal{F},\nu_X(1))[1]	\ar[r,"\alpha_X(i_*\mathcal{F})"]	\ar[d,"\cong"]			& \Hom(\Gamma_c(X_{\tau},i_*\mathcal{F}),\F_p)	\ar[d,"\cong"]	\\
   R\Hom_Z(\mathcal{F},i^!\nu_X(1))[1]						                                    	& \Hom(\Gamma_c(Z_{\tau},\mathcal{F}),\F_p)			\\
   R\Hom_Z(\mathcal{F},\F_p)		\ar[r,"\cong"]	\ar[u,"\cong"',"\text{purity}"]					& \Hom(\mathcal{F}(Z),\F_p).           \ar[u,"\cong"]
  \end{tikzcd}
 \]
 In order to show that $\alpha_X(i_*\mathcal{F})$ is an isomorphism, we are left with showing that the diagram commutes.
 In order to do so we need to unravel the definition of the map $\alpha_X(i_*\mathcal{F})$ given in (\ref{duality_Milne}).
 The crucial part for showing that the diagram commutes, which does not come from functoriality or adjunction, is the commutativity of the following diagram:
 \[
  \begin{tikzcd}
   R\Hom_k(R\bar{\pi}_!i_*\mathcal{F},R\bar{\pi}_*\nu_{\bar{X}}(1))[1] \ar[r,"tr"] \ar[d,"="] & R\Hom_k(R\pi_!i_*\mathcal{F},\F_p)  \ar[d,"="] \\
   R\Hom_k(\mathcal{F},R\bar{\pi}_*\nu_{\bar{X}}(1))[1] \ar[r,"tr"]                              & R\Hom_k(\mathcal{F},\F_p)  \ar[d,"cl"]  \\
   R\Hom_k(\mathcal{F},R\bar{\pi}_*i_*i^!\nu_{\bar{X}}(1))[1]     \ar[u,"ad"]  \ar[r,"="]        & R\Hom_k(\mathcal{F},i^!\nu_{\bar{X}}(1))[1].
  \end{tikzcd}
 \]
 The upper horizontal map is part of the definition of $\alpha_X(i_*\mathcal{F})$.
 It naturally identifies with the middle horizontal map as~$i$ is a section to~$\pi$, hence also to~$\bar{\pi}$.
 The lower part of the diagram commutes because of the subsequent \cref{compatibility_purity_trace}.
\end{proof}

\begin{lemma} \label{compatibility_purity_trace}
 Suppose that~$k$ is algebraically closed and $\pi: X \to \Spa(k,k)$ proper.
 Let $\tau \in \{t,\set\}$.
 For every section $i: \Spa(k,k) \to X$ of~$\pi$ consider the diagram
 \[
  \begin{tikzcd}
   i^!\nu_X(1)[1] = R\pi_* i_* i^! \nu_X(1)[1]  \ar[rr,"ad"]    &                                   & R\pi_* \nu_X(1)[1]   \ar[dl,"tr"]    \\
                                                                & \F_p.    \ar[ul,"\sim","cl"']
  \end{tikzcd}
 \]
 Then
 \[
  cl = tr \circ ad.
 \]
\end{lemma}

\begin{proof}
 In order to see this we have to remember that the trace map as well as the Gysin map~$cl$ are defined by the corresponding maps for the coherent sheaf~$\Omega_{X/k}^1$ on~$X$ via the short exact sequence
 \[
  0 \to \nu_X(1) \to \Omega^1_{X/k} \overset{C-1}{\longrightarrow} \Omega^1_{X/k} \to 0.
 \]
 These maps actually originate from maps of coherent sheaves on~$C$ such that the commutativity of the above diagram comes down to the commutativity of
 \[
  \begin{tikzcd}
   H^1_P(C,\Omega^1_{C/k})  \ar[rr] &                                   & H^1(C,\Omega^1_{C/k})   \ar[dl,"tr"]    \\
                                    & k,    \ar[ul,"\sim","cl"']
  \end{tikzcd}
 \]
 where we have written~$P$ for the support of the image of the section~$i$.
 But the trace map is defined to map the fundamental class of a closed point (which is independent of the chosen point) to $1 \in \F_p$ (see \cite{Gro95}, \S 4 and \S 5), whence the commutativity of the diagram.
\end{proof}

\begin{lemma} \label{duality_open_subspace}
 Let $V \to C$ be an open immersion ($V$ nonempty). Consider the induced open immersion
 \[
  j: U: = \Spa(V,k) \to X.
 \]
 Then for $\tau \in \{t,\set\}$ and every horizontally constructible sheaf of $\F_p$-modules on $X_{\tau}$ we have that $\alpha_X(\mathcal{F})$ is an isomorphism if and only if $\alpha_U(j^*\mathcal{F})$ is an isomorphism.
\end{lemma}

\begin{proof}
 Denote by $i:Z \to X$ the embedding of the (reduced) complement of~$U$.
 The short exact sequence
 \[
  0 \to j_!j^*\mathcal{F} \to \mathcal{F} \to i_*i^*\mathcal{F} \to 0
 \]
 induces the following diagram of distinguished triangles in $D^+(\F_p)$:
 \[
  \begin{tikzcd}
   R\Hom_Z(i_*i^*\mathcal{F},\F_p)[1]		\ar[r]	\ar[d,"\alpha_X(i_*i^*\mathcal{F})"]	& R\Hom_X(\mathcal{F},\nu_X(1))[1]	\ar[r]	\ar[d,"\alpha_X(\mathcal{F})"]	& R\Hom_X(j_!j^*\mathcal{F},\nu_U(1))[1]	\ar[d,"\alpha_X(j_!j^*\mathcal{F})"]	\\
   \Hom(R\Gamma_c(X_{\tau},i_*i^*\mathcal{F}),\F_p)	\ar[r]						& \Hom(R\Gamma_c(X_{\tau},\mathcal{F}),\F_p)	\ar[r]					& \Hom(R\Gamma_c(X_{\tau},j_!j^*\mathcal{F}),\F_p)
  \end{tikzcd}
 \]
 As we have seen in \cref{duality_proper_closed_subspace} that $\alpha_X(i_*i^*\mathcal{F})$ is a quasi-isomorphism, we know that $\alpha_X(\mathcal{F})$ is a quasi-isomorphism if and only if $\alpha_X(j_!j^*\mathcal{F})$ is a quasi-isomorphism.
 But the latter is a quasi-isomorphism if and only if $\alpha_U(j^*\mathcal{F})$ is a quasi-isomorphism by \cref{duality_etale_morphism}.
\end{proof}

\begin{lemma} \label{reduction_projective}
 Let $\tau \in \{t,\set\}$.
 Suppose that $\alpha_X(\mathcal{F})$ is a quasi-isomorphism whenever~$k$ is algebraically closed,~$C$ is projective and~$\mathcal{F}$ is constant.
 Then $\alpha_X(\mathcal{F})$ is a quasi-isomorphism for all smooth curves~$C$ over any perfect field~$k$ and all horizontally constructible sheaves~$\mathcal{F}$.
\end{lemma}

\begin{proof}
 \cref{reduction_alg_closed} reduces us to the case of algebraically closed fields.
 Fix a smooth curve~$C$ and a horizontally constructible sheaf~$\mathcal{F}$.
 The smooth compactification $C \to \bar{C}$ induces a compactification $j: X \to \bar{X}: \Spa(\bar{C},k)$.
 Replacing~$\mathcal{F}$ by $j_!\mathcal{F}$ and using \cref{duality_etale_morphism} we reduce to the case that~$C$ is projective.

 Now choose a Zariski-open immersion
 \[
  j: U:= \Spa(V,k) \to X
 \]
 such that $j^*\mathcal{F}$ is locally constant.
 Then there is a finite $\tau$-morphism $\pi_U:U'= \Spa(V',k) \to U$ such that $j^*\mathcal{F}|_{U'}$ is constant, i.e., it is the constant sheaf $C_{U'}$ associated to some finite $\F_p$-module~$C$.
 The smooth compactification $C'$ of~$V'$ provides a smooth compactification~$X'$ of~$U'$ and we obtain a Cartesian diagram
 \[
  \begin{tikzcd}
   U'	\ar[r,open,"j'"]	\ar[d,"\pi_U"']	& X'	\ar[d,"\pi_X"]	\\
   U	\ar[r,open,"j"]				& X
  \end{tikzcd}
 \]
 whose vertical morphisms are finite flat.
 By assumption $\alpha_{X'}(C_{X'})$ is a quasi-isomorphism.
 Then \cref{duality_open_subspace} implies that $\alpha_{U'}(C_{U'})$ is a quasi-isomorphism.
 Hence, by \cref{duality_etale_morphism} we know that $\alpha_U(\pi_{U*}C_{U'})$ is a quasi-isomorphism.
 In order to proceed we have to relate $\pi_{U*}C_{U'}$ with~$\mathcal{F}$:
 \[
  \pi_{U*}C_{U'} \cong \pi_{U*}j'^*\pi_X^*\mathcal{F} \cong j^*\pi_{X*}\pi_X^*\mathcal{F}.
 \]

 The first identification holds because we chose~$U'$ such that $\mathcal{F}|_{U'} \cong C_{U'}$.
 The second is due to base change for finite morphisms (\cref{basechange}).
 Notice that $\pi_{X*}\pi_X^*\mathcal{F}$ is constructible by \cref{pushforward_constructible}.
 Therefore, we can apply \cref{duality_open_subspace} to $\pi_{X*}\pi_X^*\mathcal{F}$ and the open embedding~$j$ to conclude that $\alpha_X(\pi_{X*}\pi_X^*\mathcal{F})$ is a quasi-isomorphism.

 The adjunction morphism $\mathcal{F} \to \pi_*\pi^*\mathcal{F}$ is injective because~$\pi$ is finite (this can be checked on stalks using \cref{basechange}).
 We define the sheaf~$\mathcal{F}'$ to be its cokernel, i.e., we have an exact sequence
 \[
  0 \to \mathcal{F} \to \pi_*\pi^*\mathcal{F} \to \mathcal{F}' \to 0.
 \]
 It induces a commutative diagram of long exact sequences
 \[
 \begin{tikzcd}[column sep=small]
  \ldots	\ar[r]	& \Ext_{\F_p}^i(\mathcal{F}',\nu_X(1))	\ar[r]	\ar[d,"\alpha_X^i(\mathcal{F}')"]	& \Ext_{\F_p}^i(\pi_*\pi^*\mathcal{F},\nu_X(1))	\ar[r]	\ar[d,"\alpha_X^i(\pi_*\pi^*\mathcal{F})"]	& \Ext_{\F_p}^i(\mathcal{F},\nu_X(1))	\ar[r]	\ar[d,"\alpha_X^i(\mathcal{F})"]	& \ldots	\\
  \ldots	\ar[r]	& H^{1-i}(X_{\tau},\mathcal{F}')^*		\ar[r]						& H^{1-i}(X_{\tau},\pi_*\pi^*\mathcal{F})^*		\ar[r]							& H^{1-i}(X_{\tau},\mathcal{F})^*		\ar[r]					& \ldots
 \end{tikzcd}	
 \]
 For $i \ge 2$ all groups involved are zero by \cref{cd,cd_ext}.
 Hence we can proceed by descending induction and assume that $\alpha_X^j(\mathcal{G})$ is an isomorphism for all $j > i$ and all horizontally constructible sheaves~$\mathcal{G}$.
 (Of course this is a bit of an overkill as we only have to treat the cases $i=0$ and $i=1$ but still makes the exposition shorter.)
 Using that $\alpha_X^i(\pi_*\pi^*\mathcal{F})$ is an isomorphism the five lemma implies that $\alpha_X^i(\mathcal{F})$ is surjective.
 This is true for any horizontally constructible sheaf, in particular for~$\mathcal{F}'$, so $\alpha_X^i(\mathcal{F}')$ is surjective, as well.
 Applying the five lemma once more, we obtain that $\alpha_X^i(\mathcal{F})$ is an isomorphism.
\end{proof}

\begin{theorem}
 Let~$C$ be a smooth curve over a perfect field~$k$ and set $X= \Spa(C,k)$.
 For $\tau \in \{t,\set\}$ and every horizontally constructible sheaf of $\F_p$-modules~$\mathcal{F}$ on~$X_{\tau}$ the homomorphism
 \[
  \alpha_X(\mathcal{F}): R\Hom_X(\mathcal{F},\nu_X(1))[1] \to R\Hom_k(R\pi_!\mathcal{F},\F_p)
 \]
 in $D^+(\F_p)$ is a quasi-isomorphism.
\end{theorem}

\begin{proof}
 By \cref{reduction_projective} we may assume that~$k$ is algebraically closed, $C$ is projective and~$\mathcal{F}$ is constant.
 We can even assume $\mathcal{F} = \F_p$.
 In this case $\alpha_X(\F_p)$ coincides with Milne's duality (\cref{alpha=Milne}), which is a quasi-isomorphism by \cite{Milne86}, Theorem~1.4.
\end{proof}

\begin{corollary}
 In the situation of the theorem assume that~$k$ is algebraically closed.
 Then the Yoneda pairing together with the trace map induce a perfect pairing
 \[
  H^i_c(X_{\tau},\mathcal{F}) \times \Ext^{1-i}(\mathcal{F},\nu_X(1)) \longrightarrow \F_p.
 \]
 for all $i \in \Z$.
\end{corollary}

\begin{remark}
 Poincar\'e duality also holds for the \'etale topology, see \cite{Mos99}.
 Actually the same proof as ours would work for the \'etale topology using the purity result obtained in loc.\ cit., \S 2.
 We just have to note that the proof of Poincar\'e duality needs purity only for $\nu(1)$ and not for $\nu(0) = \F_p$.
\end{remark}

\bibliographystyle{alpha}
\bibliography{../citations}

\end{document}